\documentclass[10pt]{article}

\usepackage[margin=1in]{geometry}
\usepackage{amsmath,amssymb,amsthm}
\usepackage{mathtools}
\usepackage[hyphens]{url}
\usepackage{xurl}
\usepackage[hidelinks]{hyperref}
\usepackage{tikz}
\usetikzlibrary{arrows,calc,fit,intersections,positioning}
\mathtoolsset{showonlyrefs}
\numberwithin{equation}{section}

\newtheorem{definition}{Definition}[section]
\newtheorem{proposition}[definition]{Proposition}
\newtheorem{theorem}[definition]{Theorem}
\newtheorem{remark}[definition]{Remark}
\newtheorem{lemma}[definition]{Lemma}
\newtheorem{corollary}[definition]{Corollary}
\newtheorem{example}[definition]{Example}



\begin{document}

\title{Similarity-Sensitive Entropy under Representation Change and Inference}

\author{Joseph Samuel Miller\\Independent Researcher\\USA}
\date{\today}

\maketitle

\begin{abstract}
Similarity-sensitive entropy measures the uncertainty of a probability law relative to a similarity kernel that encodes the distinguishability of states. We develop the measure-theoretic treatment covering both finite similarity matrices and general probability spaces, and study how the law and similarity kernel transform under measurable maps, channels, and conditioning operations. This yields data-processing inequalities, so a reduction in entropy quantifies how much distinguishability is lost under representation change. We also define a conditional similarity-sensitive entropy theory, give a counterexample to a recent conjecture on concavity, and identify a practical one-dimensional Laplace kernel class where concavity holds.
\end{abstract}


\section{Introduction}

Many information-theoretic functionals treat distinct states as perfectly distinguishable.
In applications, however, state labels often encode objects---records, strings, phenotypes, signals---for which distinct labels can be partially redundant.
A natural way to model this redundancy is to accompany the state space with a similarity kernel $K$ that quantifies how interchangeable two states are for the task at hand (e.g.\ derived from confusion/utility or from geometry via distance). Given $K$ over a state space with some probability law, we can calculate an associated entropy.

As background, our work builds on Leinster and Cobbold's similarity-sensitive diversities and entropies for finite sets~\cite{leinster_cobbold}; see also~\cite[Ch.~4]{leinster_entropy_diversity}. Leinster and Roff extended this framework to general measure spaces~\cite{leinster_roff_maxent} focusing on entropy maximization, and for broader diversity background see~\cite{hill_diversity,patil_taillie,rao_quadratic_entropy}. We focus on the order-$q=1$ case, which recovers Shannon entropy~\cite{shannon} on a discrete space when distinct states have zero similarity (cf.\ R\'enyi and Tsallis entropies~\cite{renyi,tsallis}).

Beyond foundations, we primarily develop how this entropy behaves under two operations. Here, we'll motivate them.

The first operation we call \emph{representation change}: one derives from $X$ a new variable on another space, either deterministically via a measurable map ("coarse-graining") or randomly via a Markov kernel. If the change mixes distinct states together, it may lose distinctions that mattered for the original task. In classical Shannon theory, one pushes the law forward and implicitly "resets" distinctions on the output space to the identity kernel. Deterministic coarse-graining may then lower entropy, noisy channels may raise or lower it, and in continuous settings even coordinate changes can alter differential entropy. Once a task kernel $K$ has been specified in the input space, this no longer appears satisfactory. If the output variable is to represent the original task, then one should transport not only the probability law but also the task's similarity relationships. The resulting output entropy should reflect the task-relevant distinctions that remain after the transformation; an entropy decrease should correspond to an irreversible loss of distinctions about $X$ and an increase in distinctions relevant to the original task should be impossible.

In the second operation, which we'll refer to as \emph{inference}, one observes $Y$ about a fixed task variable $X$ with associated similarity kernel $K$. Here the task itself does not change: the kernel $K$ remains associated with $X$, and only the law is updated from $\mu_X$ to the posterior law of $X\mid Y$. Entropy changes are therefore interpreted as uncertainty changes about the same task. In Shannon theory, expected entropy decreases under conditioning by concavity. For similarity-sensitive entropy, concavity of $H_K$ is no longer guaranteed and depends on the kernel.

Thus we are led to two questions. Under representation change, how should the original task kernel be moved through the transformation? Under inference, what can we say about the kernels for which conditioning reduces expected similarity-sensitive entropy?

We provide three main contributions:
\begin{itemize}
\item \textbf{Similarity transport and data processing.}
In both discrete and continuous settings, we show that entropy cannot
increase under a map or channel if each pair of outputs is at least as
similar as the inputs that produced them. We construct
the smallest kernel with this property, which in finite spaces takes the
maximum over compatible input similarities. With an identity kernel on
discrete inputs, distinct outputs of a deterministic map come from disjoint
groups of inputs. The output kernel is therefore also the identity, and our
result recovers Shannon's coarse-graining inequality $H(f(X))\le H(X)$.
\item \textbf{Conditional $K$-entropy and inference.}
We define a conditional $K$-entropy and $K$-information gain, recover Shannon conditional quantities for partition kernels, identify a useful one-dimensional Laplace pullback class where concavity holds, and give a counterexample showing that symmetric positive definiteness (SPD) plus the multiplicative triangle inequality (MTI) does not imply concavity in general.
\item \textbf{Measure-theoretic formulation and finite approximation.}
We give a uniform representation on $([0,1],\lambda)$ and a kernel/matrix approximation scheme, connecting the measure-theoretic constructions to finite similarity matrices.
\end{itemize}

We work on standard Borel measurable spaces.
Similarity kernels are measurable on product $\sigma$-algebras, symmetric, and $[0,1]$-valued with unit diagonal.
For a fixed law $\mu$, comparisons between kernels on that space are
understood $(\mu\otimes\mu)$-almost everywhere unless stated pointwise.
All entropy quantities considered below, including averaged conditional
entropies, are assumed finite.

\section{Similarity-Sensitive Entropy}
\label{sec:similarity-sensitive-entropy}

This section defines similarity-sensitive entropy on general probability spaces, specializes to the finite case, and provides the partition-kernel and representation ideas used later.

\subsection{General kernelled probability spaces}
\label{sec:general-kernels}

\begin{definition}[Kernel on a measurable space]
Let $(\Omega,\mathcal{F})$ be a measurable space.
A \emph{similarity kernel} on $\Omega$ is a map
\begin{equation}
  K: \Omega\times\Omega \to [0,1]
\end{equation}
such that:
\begin{enumerate}
  \item $K$ is measurable with respect to $\mathcal{F}\otimes\mathcal{F}$;
  \item $K(\omega,\omega') = K(\omega',\omega)$ for all $\omega,\omega'$;
  \item $K(\omega,\omega) = 1$ for all $\omega$;
\end{enumerate}
\end{definition}

\begin{definition}[The typicality function]
Let $(\Omega,\mathcal{F},\mu)$ be a probability space and let $K$ be a similarity kernel on $\Omega$.
The \emph{typicality} of $\omega\in\Omega$ is its expected similarity to a randomly drawn state $W\sim\mu$:
\begin{equation}
  \tau(\omega) := \mathbb{E}[K(\omega,W)]
  = \int_\Omega K(\omega,\omega')\, d\mu(\omega').
\end{equation}
\end{definition}

Since $0\le K\le 1$ and $\mu$ is a probability measure, $\tau(\omega)\in[0,1]$ for all $\omega$.

\begin{definition}[Similarity-sensitive entropy on a probability space]
Let $(\Omega,\mathcal{F},\mu)$ be a probability space and let $K$ be a similarity kernel on $\Omega$, with typicality function $\tau$.
The \emph{$K$-entropy} of $\mu$ is
\begin{equation}
  H_K(\mu)
  := \int_\Omega \bigl(-\log \tau(\omega)\bigr)\, d\mu(\omega),
  \label{eq:general-K-entropy}
\end{equation}
\end{definition}

For a random variable $X$ with law $\mu_X$, we also write $H_K(X):=H_K(\mu_X)$.

The value of $H_K(\mu)$ depends only on the distribution of the typicality function
$\tau(\omega)$ under $\omega\sim\mu$.
In particular, if $K'\ge K$, then $H_K(\mu)\ge H_{K'}(\mu)$ because
enlarging a kernel increases typicality.

\subsection{Discrete similarity-sensitive entropy}
\label{sec:discrete-entropy}

Let $\mathsf{X}$ be a finite set, and let $p_x$ be the probability of $x\in\mathsf{X}$.
Write $p$ for the vector of these probabilities. Let $K$ be a similarity matrix:
$K\in[0,1]^{\mathsf{X}\times\mathsf{X}}$ is symmetric and $K_{x,x}=1$.

\begin{definition}[Typicality and discrete $K$-entropy]
For a randomly drawn state $X'\sim p$, the typicality of $x$ is
\begin{equation}
  \tau(x) := \mathbb{E}[K_{x,X'}] = (Kp)_x.
\end{equation}
Thus the \emph{typicality vector} is $Kp$.
The \emph{discrete $K$-entropy} is
\begin{equation}
    H_K(p)
    := - \sum_{x\in \mathsf{X}} p_x \log (Kp)_x,
    \label{eq:discrete-K-entropy}
\end{equation}
and for a random variable $X\sim p$ we write $H_K(X):=H_K(p)$.
Terms with $p_x=0$ contribute zero.
\end{definition}

For finite $\Omega=\mathsf{X}$ with probability vector $p$, the general typicality
is $\tau(x)=(Kp)_x$, and \eqref{eq:general-K-entropy} reduces to
\eqref{eq:discrete-K-entropy}.

If $p_x>0$, then
\begin{equation}
0<(Kp)_x\le 1,
\end{equation}
since $(Kp)_x\ge p_x$ and $(Kp)_x\le\sum_{x'}1\cdot p_{x'}=1$.
Thus $H_K(p)$ is always well-defined in the discrete case.
If $K=I$, then $H_K(p)=H(p)$ is the Shannon entropy. Also a consequence of $(Kp)_x\ge p_x$, we have $H_K(p)\le H(p)$ so any similarity between states with positive probability reduces the entropy.

\subsection{Partition kernels and coarse variables}

A partition kernel treats states in the same group as identical and states in
different groups as mutually distinct.

\begin{definition}[Finite-class partition kernel]
\label{def:finite-partition-kernel}
Let $\Omega$ be a measurable space. A \emph{finite-class partition kernel}
assigns similarity $1$ to states in the same group and $0$ to states in
different groups, for a finite partition of $\Omega$ into measurable groups.
Label the groups $1,\dots,m$, and let $\pi(\omega)$ be the group containing
$\omega$. Then
\[
K(\omega,\omega')=\mathbf{1}\{\pi(\omega)=\pi(\omega')\}.
\]
For a random state $X$, the associated \emph{coarse variable} is $Z=\pi(X)$,
the group containing $X$.
\end{definition}

\begin{proposition}[Partition kernels recover Shannon entropy]
\label{prop:partition-matrix-entropy}
Let $K$ be a finite-class partition kernel with group map $\pi$, and let
$Z=\pi(X)$. Then
\begin{equation}
  H_K(X) = H(Z),
\end{equation}
where $H(Z)$ is the Shannon entropy of $Z$.
\end{proposition}

\begin{proof}
Write $q_j:=\mathbb{P}(Z=j)$. The typicality of a state in group $j$ is
$q_j$, the probability of drawing another state from that group. Hence,
with zero-probability groups contributing zero,
\[
H_K(X)=-\sum_j q_j\log q_j=H(Z).
\]
\end{proof}

\subsection{Relabeling and pullbacks}
\label{sec:isomorphisms}

\begin{definition}[Isomorphism]
An \emph{isomorphism} between $(\Omega,\mu,K)$ and $(\Omega',\mu',K')$
is a measurable bijection $\phi$, up to null sets, that preserves the
probability law and satisfies
\[
  K'(\phi(x),\phi(x'))=K(x,x') \quad\text{under the input law }\mu.
\]
\end{definition}

\begin{lemma}[Pullback identity]
\label{lem:pullback-identity-any-output-kernel}
Let $X$ have law $\mu$ and set $Y=f(X)$ for a measurable map $f$. For an
output similarity kernel $L$, define its pullback by
$L^f(x,x')=L(f(x),f(x'))$. Then
\[
\tau_{L^f}(x)=\tau_L(f(x)),\qquad H_{L^f}(X)=H_L(Y).
\]
\end{lemma}

\begin{proof}
For a fresh draw $X'\sim\mu$, the variable $f(X')$ has the law of $Y$.
Thus $\tau_{L^f}(x)=\mathbb E[L(f(x),f(X'))]=\tau_L(f(x))$.
Taking the expected negative log gives the entropy identity.
\end{proof}

\begin{proposition}[Invariance under isomorphism]
\label{prop:invariance-isomorphism}
Isomorphic spaces with similarity kernels have the same typicality
distribution and hence the same entropy.
\end{proposition}

\begin{proof}
For $X\sim\mu$, an isomorphism $\phi$ satisfies $\phi(X)\sim\mu'$.
The pullback identity and $K=(K')^\phi$ give the same distribution for
$\tau_K(X)$ and $\tau_{K'}(\phi(X))$.
\end{proof}

\section{Similarity under representation change}
\label{sec:posterior-induced-output-kernels}

For an invertible relabeling $Y=f(X)$, we can transport similarity directly
by setting $L(y,y')=K(f^{-1}(y),f^{-1}(y'))$ and preserve entropy. When
several inputs share an output, we need a rule for combining their
similarities. We treat each output as a state in its own right and assign it self-similarity \(1\).

For distinct outputs, we think of similarity as reflecting shared aspects between the states.
An output inherits aspects of the inputs that could have produced it, so
we require each output pair to be at least as similar as its corresponding
input pair. In finite spaces, the smallest such similarity is the maximum
over compatible input pairs.

This rule also gives a data-processing inequality. For an identity input kernel,
deterministic coarse-graining yields an identity output kernel and recovers
Shannon's coarse-graining inequality. We first establish the inequality, then
construct the smallest kernel satisfying the similarity requirement.

\subsection{A sufficient condition for data processing}

Let $X$ have similarity kernel $K$ and let $Y$ be its new representation.
Their joint law specifies which inputs can produce which outputs, whether
$Y$ is a deterministic function of $X$ or comes from a randomized channel.

\begin{theorem}[Data processing for a joint law]
\label{thm:joint-law-dpi}
Let $(X,Y)$ and $(X',Y')$ be independent draws from the joint law.
If an output similarity kernel $L$ satisfies
\begin{equation}
  L(Y,Y')\ge K(X,X') \quad\text{almost surely},
  \label{eq:transport-domination}
\end{equation}
then
\[
  H_L(Y)\le H_K(X).
\]
\end{theorem}

\begin{proof}
Take conditional expectations in \eqref{eq:transport-domination} given
$(X,Y)$. By independence of the two draws,
\[
  \tau_L(Y)
  = \mathbb{E}[L(Y,Y')\mid X,Y]
  \ge \mathbb{E}[K(X,X')\mid X,Y]
  = \tau_K(X).
\]
Take negative logarithms and average over $(X,Y)$ to get the inequality.
\end{proof}

\begin{corollary}[When representation change preserves entropy]
\label{cor:entropy-preservation}
Suppose $L$ satisfies \eqref{eq:transport-domination}. Then
\[
H_L(Y)=H_K(X)
\quad\Longleftrightarrow\quad
L(Y,Y')=K(X,X')\ \text{almost surely}.
\]
\end{corollary}

\begin{proof}
The change in typicality is the conditional average of the nonnegative change in similarity. If similarity increases with positive probability, so does typicality, and taking negative logs and averaging would give \(H_L(Y)<H_K(X)\). If the similarities agree almost surely, the typicalities agree and the entropies are equal.
\end{proof}

Corollary~\ref{cor:entropy-preservation} extends the invariance principle of
Proposition~\ref{prop:invariance-isomorphism} beyond invertible relabelings.
For example, replacing $X$ by its group $Z$ under a partition kernel preserves entropy
because the group labels determine all the input similarities.

For a deterministic map $Y=f(X)$, condition~\eqref{eq:transport-domination}
requires that $L^f\ge K$ under the input law.
An accompanying essay uses the pointwise equality $L^f=K$ as a criterion
for choosing features that preserve task similarities~\cite{miller_feature_engineering}.

Here data processing concerns entropy after transporting the similarity
kernel. Classical data-processing inequalities concern quantities such as
mutual information~\cite{cover_thomas} and $f$-divergences~\cite{ali_silvey}.
Section~\ref{sec:conditional-mutual} gives a separate data-processing
inequality for information gain on a fixed input space.

\subsection{The induced output kernel}

For each pair of distinct outputs, the induced kernel gives the smallest
upper bound on the similarities of compatible input pairs in the sense of
\eqref{eq:transport-domination}. Write $\mu_{XY}$ for the joint law and
$\mu_{X\mid Y=y}$ for the conditional law of the input given $y$.

\begin{definition}[Induced output kernel]
\label{def:posterior-induced-output-kernel}
Set $K^{\mathsf{Y},\mu}(y,y)=1$. For distinct outputs, define
\begin{equation}
  K^{\mathsf{Y},\mu}(y,y')
  := \operatorname*{ess\,sup}_{(x,x')\sim
      \mu_{X\mid Y=y}\otimes\mu_{X\mid Y=y'}} K(x,x').
\end{equation}
The induced kernel is defined up to $\mu_Y\otimes\mu_Y$-almost everywhere equality.
\end{definition}

In finite spaces, the induced similarity of two distinct outputs is the
largest input similarity across their two conditional supports. This applies
to deterministic maps and randomized channels alike.

\begin{example}[Three-state coarse-graining]
Take three equally spaced states $0,1,2$, each with positive probability,
and halve the similarity at each step, so $K_{ij}=2^{-|i-j|}$.
If we merge $0$ and $1$ into output $A$ and send $2$ to output $B$. The input and
output similarity matrices, ordered as $(0,1,2)$ and $(A,B)$, are
\[
  K=\begin{pmatrix}
    1 & \tfrac12 & \tfrac14\\
    \tfrac12 & 1 & \tfrac12\\
    \tfrac14 & \tfrac12 & 1
  \end{pmatrix},
  \qquad
  K_{\mathrm{out}}=\begin{pmatrix}
    1 & \tfrac12\\
    \tfrac12 & 1
  \end{pmatrix}.
\]
The similarity between $A$ and $B$ is $\max\{K_{02},K_{12}\}=\tfrac12$,
the smallest value that bounds both input similarities.
\end{example}

\begin{proposition}[Minimality of the induced output kernel]
\label{prop:posterior-induced-fixed-law-minimality}
The induced output kernel $K^{\mathsf{Y},\mu}$ is measurable and satisfies
\eqref{eq:transport-domination}. Every output kernel $L$ satisfying that same
condition obeys
$K^{\mathsf{Y},\mu}\le L$ under the output law $\mu_Y$.
\end{proposition}

\begin{proof}
Appendix~\ref{sec:appendix-measurability-induced-kernel} establishes
measurability and independence of the choice of conditional laws.
Use $\mu_{X\mid Y=y}\otimes\mu_{X\mid Y=y'}$ as the conditional law
of the input pair given $(Y,Y')=(y,y')$.
Condition~\eqref{eq:transport-domination} holds exactly when $L(y,y')$
bounds $K(X,X')$ almost surely under this conditional law for almost every
output pair. For distinct outputs, the essential supremum is
exactly the smallest such bound. On the diagonal, every output similarity
kernel is required to take the value $1$, which also bounds $K$.
\end{proof}

Since increasing similarities decreases entropy, the induced kernel
retains the most entropy among output kernels satisfying
\eqref{eq:transport-domination}. Under an
isomorphism, each conditional input law is a point at the inverse, so the induced kernel is the original kernel relabeled. Appendix~\ref{sec:nuisance-invariance} gives a condition under
which appending a nuisance variable leaves the induced entropy unchanged.

Figure~\ref{fig:fiberwise_max} illustrates the induced kernel and its pullback.

\begin{figure}[t]
  \centering
		  \begin{tikzpicture}[
		    >=latex,
		    font=\footnotesize,
		    line cap=round,
        fiber/.style={draw=black!70, rounded corners=4pt, fill=black!3, thick, minimum width=3.0cm, minimum height=2.0cm},
		    point/.style={circle, fill=black, inner sep=1.1pt},
		    sim/.style={draw=black!35, line width=0.45pt},
		    max/.style={draw=black, line width=1.1pt},
		    map/.style={->, draw=black!45, line width=0.55pt, shorten >= 3pt, shorten <= 3pt},
		    labelbox/.style={fill=white, fill opacity=0.95, text opacity=1, inner sep=2pt, rounded corners=2pt},
		  ]

		    % Layout anchors: top Omega, middle Y, bottom Omega (back-composed)
		    \coordinate (topL) at (0, 3.4);
		    \coordinate (topR) at (5.3, 3.4);
		    \coordinate (midL) at (0, 0.8);
		    \coordinate (midR) at (5.3, 0.8);
		    \coordinate (midC) at ($(midL)!0.5!(midR)$);
		    \coordinate (botL) at (0, -1.8);
		    \coordinate (botR) at (5.3, -1.8);

	    % Helper: vertical center line for label alignment.
	    \path[name path=centerVert] ($(midC)+(0,6)$) -- ($(midC)+(0,-6)$);

		    % Space labels
		    \node[anchor=east, font=\large] at ($(topL)+(-2.6,0.00)$) {$\Omega$};
		    \node[anchor=east, font=\large] at ($(midL)+(-2.6,0.00)$) {$\mathsf{Y}$};
		    \node[anchor=east, font=\large] at ($(botL)+(-2.6,0.00)$) {$\Omega$};
	    \node[anchor=east, font=\scriptsize, black!70, align=right] at ($(topL)+(-2.6,-0.45)$) {$K:\Omega\times\Omega\to[0,1]$};
	    \node[anchor=east, font=\scriptsize, black!70, align=right] at ($(midL)+(-2.6,-0.45)$) {$K^{\mathsf{Y},\mu}:\mathsf{Y}\times\mathsf{Y}\to[0,1]$};
	    \node[anchor=east, font=\scriptsize, black!70, align=right] at ($(botL)+(-2.6,-0.45)$) {$K^{f,\mu}:\Omega\times\Omega\to[0,1]$};

    % --- Top: original kernel on Omega, with ess sup across fibers ---
    \node[fiber] (F1) at (topL) {};
    \node[fiber] (F2) at (topR) {};
    \node[anchor=south, font=\small] at (F1.north) {$f^{-1}(y)$};
    \node[anchor=south, font=\small] at (F2.north) {$f^{-1}(y')$};

    % Points are representatives in each fiber.
    \node[point] (x1) at ($(F1.center)+(-1.20,0.85)$) {};
    \node[point] (x2) at ($(F1.center)+(-0.35,-0.10)$) {};

	    \node[point] (u1) at ($(F2.center)+(1.10,0.75)$) {};
	    \node[point] (u2) at ($(F2.center)+(0.35,-0.15)$) {};
	    \node[point] (u3) at ($(F2.center)+(-0.95,-0.70)$) {};

	    % Similarity links (schematic: lighter/thinner = smaller similarity)
	    \draw[sim, draw=black!20, line width=0.35pt] (x1) -- (u2);
	    \draw[sim, draw=black!45, line width=0.55pt] (x1) -- (u3);

		    \draw[max, name path=essLine] (x2) -- (u1);
		    \path[name intersections={of=essLine and centerVert, by={essLabel}}];
		    \node[above=7pt, labelbox, font=\scriptsize] at (essLabel) {$\operatorname*{ess\,sup} K(\omega,\omega')$};

    % --- Middle: induced kernel on Y ---
		    \node[point] (yy) at (midL) {};
		    \node[point] (yyp) at (midR) {};
		    \node[font=\scriptsize, anchor=east, xshift=-6pt] at (yy) {$y$};
		    \node[font=\scriptsize, anchor=west, xshift=6pt] at (yyp) {$y'$};
	    \draw[max] (yy) -- (yyp)
	      node[midway, above=2pt, labelbox]
	      {$K^{\mathsf{Y},\mu}(y,y')$};

    % f: Omega -> Y
    \draw[map] (F1.south) -- (yy.north);
    \draw[map] (F2.south) -- (yyp.north);
    \node[labelbox, font=\scriptsize, text=black!70] at ($(midC)+(0,0.95)$) {$f:\Omega\to\mathsf{Y}$};

    % --- Bottom: pullback/back-composition of the induced kernel ---
    \node[fiber] (G1) at (botL) {};
    \node[fiber] (G2) at (botR) {};
    % Label the fibers outside the back-composed boxes to avoid overlap with the schematic links.
    \node[anchor=north, font=\small, yshift=-1pt] at (G1.south) {$f^{-1}(y)$};
    \node[anchor=north, font=\small, yshift=-1pt] at (G2.south) {$f^{-1}(y')$};

    % Use the same representative points/pairs as above (emphasizing that $K^{f,\mu}$ rewrites their similarities).
    \node[point] (a1) at ($(G1.center)+(-1.20,0.85)$) {};
    \node[point] (a2) at ($(G1.center)+(-0.35,-0.10)$) {};
    \node[point] (b1) at ($(G2.center)+(1.10,0.75)$) {};
    \node[point] (b2) at ($(G2.center)+(0.35,-0.15)$) {};
    \node[point] (b3) at ($(G2.center)+(-0.95,-0.70)$) {};

    % Back-composition makes the similarity constant across the fiber block:
    % draw several representative pairs (not one-to-one) with identical styling.
	    \draw[max] (a1) -- (b2);
	    \draw[max] (a1) -- (b3);
	    \draw[max, name path=pullLine] (a2) -- (b1);
	    \path[name intersections={of=pullLine and centerVert, by={pullLabel}}];
	    \node[above=7pt, labelbox, fill opacity=0.35, inner sep=1.0pt, font=\scriptsize] at (pullLabel) {$K^{f,\mu}(\omega,\omega')$};

    % Pullback from Y to Omega (schematic)
    \draw[map] (yy.south) -- (G1.north);
    \draw[map] (yyp.south) -- (G2.north);
    \node[labelbox, align=center, font=\scriptsize, text=black!70] at ($(midC)+(0,-0.85)$) {$L\mapsto L^f$\\(pullback)};

  \end{tikzpicture}
  \caption{Given a similarity kernel $K:\Omega\times\Omega\to[0,1]$ and a measurable map $f:\Omega\to\mathsf{Y}$, the induced output kernel $K^{\mathsf{Y},\mu}$ is defined by the fiberwise essential supremum of $K$ (lighter lines indicate smaller similarity). Pulling $K^{\mathsf{Y},\mu}$ back along $f$ yields the back-composed kernel $K^{f,\mu}=(K^{\mathsf{Y},\mu})^f$ on $\Omega$, which is constant on fiber blocks.}
  \label{fig:fiberwise_max}
\end{figure}

\subsection{A law independent rule for deterministic maps}

The induced kernel depends on the input law. If we istead want to know the output kernel
that works for any input law, we can restrict to deterministic maps and must account for every input pair.

\begin{theorem}[Law-independent rules for deterministic maps]
\label{thm:minimality-general}
Given an input similarity kernel $K$, a measurable map $f$, and an output
similarity kernel $L$. The inequality $H_L(f_\#\mu)\le H_K(\mu)$ holds
for every input law $\mu$ if and only if
\[
  L(f(x),f(x'))\ge K(x,x')
  \quad\text{for every input pair }x,x'.
\]
\end{theorem}

\begin{proof}
The pointwise bound implies \eqref{eq:transport-domination} under every
input law, so Theorem~\ref{thm:joint-law-dpi} applies.

Now suppose the entropy inequality holds for every input law. Give two
inputs $x,x'$ with distinct outputs probability $1/2$ each. Their outputs
$f(x),f(x')$ then also have probability $1/2$ each. The input and output
entropies are
\[
  H_K(\mu)=\log\frac{2}{1+K(x,x')},
  \qquad
  H_L(f_\#\mu)=\log\frac{2}{1+L(f(x),f(x'))}.
\]
Comparing the denominators, the entropy inequality gives
$L(f(x),f(x'))\ge K(x,x')$. For inputs with the same output, the bound
follows from the unit diagonal of $L$.
\end{proof}

When the input space has finitely many states, the smallest kernel that works for every law assigns each pair of attained outputs the maximum similarity over all input pairs that map to them. For a fixed law, the induced kernel takes the maximum over
input pairs of positive probability.

\section{Conditional Similarity-Sensitive Entropy and Information Gain}
\label{sec:conditional-mutual}

In inference, we observe $Y$ to learn about $X$. The kernel $K$ tells us which distinctions between values of $X$ matter. We use the same $K$ to compare the prior entropy of $X$ with its posterior entropy after observing $Y$, and may refer to $(X,K)$ as the task variable/kernel to distinguish it from the observed variable $Y$.

\subsection{General conditional \texorpdfstring{$K$}{K}-entropy}

\begin{definition}[Conditional $K$-entropy of $X$ given $Y$]
Let $K$ be the similarity kernel on the state space of $X$. For almost
every observation $y$, set $H_K(X\mid Y=y)=H_K(\mu_{X\mid Y=y})$.
Define the \emph{conditional $K$-entropy} by averaging these entropies.
\[
  H_K(X\mid Y)=\mathbb E[H_K(\mu_{X\mid Y})].
\]
\end{definition}

\begin{definition}[$K$-information gain about $X$]
\label{def:mutual-info-K}
We define the expected $K$-information gain
about $X$ from observing $Y$ by
\begin{equation}
  I_K(X;Y) := H_K(X) - H_K(X\mid Y).
\end{equation}
\end{definition}

A sufficient condition for $I_K(X;Y)\ge 0$ is concavity of $\mu\mapsto H_K(\mu)$: since $\mu_X$ is the $\mu_Y$-mixture of posteriors $\mu_{X\mid Y=y}$, the mixture inequality in
Lemma~\ref{lem:posterior-mixture-jensen} gives
$H_K(X\mid Y)\le H_K(X)$.
We call this \emph{inference monotonicity} (conditioning cannot increase expected $K$-entropy).
Conversely, if $\mu\mapsto H_K(\mu)$ is not concave, a two-state mixture construction (realized by a binary $Y$) can yield a joint law for which $H_K(X\mid Y)>H_K(X)$.

\subsection{Partition kernels as the exact Shannon case}

For a finite-class partition kernel, the identity
$H_K(X)=H(Z)$ extends to conditional entropy: only the group containing $X$
matters, even after observing $Y$.

\begin{proposition}[Conditional entropy for partition kernels]
\label{prop:conditional-partition-kernel}
Assume $K$ is a finite-class partition kernel on $\Omega_X$, with associated map $\pi$
and coarse variable $Z=\pi(X)$. For any joint law of $(X,Y)$,
\begin{equation}
  H_K(X\mid Y) = H(Z\mid Y),
\end{equation}
where $H(Z\mid Y)$ is the usual Shannon conditional entropy. Consequently,
\begin{equation}
  I_K(X;Y)=I(Z;Y)\ge 0.
\end{equation}
\end{proposition}

\begin{proof}
Proposition~\ref{prop:partition-matrix-entropy} gives $H_K(X)=H(Z)$.
Apply the same identity to the conditional law of $X$ given $Y$ and average
over $Y$ to obtain $H_K(X\mid Y)=H(Z\mid Y)$, and hence
$I_K(X;Y)=I(Z;Y)\ge 0$.
\end{proof}

\subsection{Concavity for one-dimensional Laplace pullback kernels}
\label{subsec:conditional-laplace-concavity}

On $\mathbb{R}$, define the Laplace similarity kernel
\begin{equation}
  K_{\mathrm{Lap}}(s,t):=\exp(-|s-t|).
\end{equation}
More generally, given a measurable map $g:\Omega_X\to\mathbb{R}$, define the \emph{one-dimensional Laplace pullback kernel}
\begin{equation}
  K_{\mathrm{Lap}}^g(\omega,\omega')
  := K_{\mathrm{Lap}}(g(\omega),g(\omega'))
  = \exp(-|g(\omega)-g(\omega')|).
\end{equation}

\begin{theorem}[Concavity for the Laplace similarity kernel]
\label{thm:laplace-concavity}
The functional $\mu\mapsto H_{K_{\mathrm{Lap}}}(\mu)$ is concave on
probability laws on $\mathbb R$ satisfying the standing entropy assumption.
\end{theorem}

\begin{proof}[Proof idea]
On an ordered grid, the inverse Laplace matrix has nonpositive off-diagonal
entries and positive row sums. These signs let us write negative entropy as
a sum of convex functions of the typicalities. We then approximate the
laws on common finite grids and use dominated convergence to pass the
concavity inequality to the limit.
Details are in Appendix~\ref{sec:appendix-laplace-concavity-proof}.
\end{proof}

\begin{corollary}[Inference for Laplace pullback kernels]
\label{cor:laplace-pullback-inference}
If $K=K_{\mathrm{Lap}}^g$ for a measurable feature $g$, then
$H_K(X\mid Y)\le H_K(X)$ and $I_K(X;Y)\ge0$.
If $\mathbb E[|g(X)|]<\infty$, then both $H_K(X)$ and $H_K(X\mid Y)$ are
finite.
\end{corollary}

\begin{proof}
The pullback identity gives $H_K(\mu)=H_{K_{\mathrm{Lap}}}(g_\#\mu)$.
Pushforward commutes with mixtures, so Theorem~\ref{thm:laplace-concavity}
gives concavity and Lemma~\ref{lem:posterior-mixture-jensen} applies.
For two independent draws $S,S'$ from any one-dimensional law, Jensen's
inequality gives $H_{K_{\mathrm{Lap}}}(S)\le\mathbb E|S-S'|\le2\mathbb E|S|$.
Apply this bound to $g(X)$ and to its conditional laws, then average
over $Y$ for the latter.
\end{proof}

\begin{corollary}[Mutual-information data-processing inequality under concavity]
\label{cor:mi-dpi-concave}
Let $(\Omega_X,\mathcal{F}_X,\mu_X,K)$ be a kernelled probability space such that $\mu\mapsto H_K(\mu)$ is concave on the relevant class of laws.
If $X\to Y\to Z$ is a Markov chain (i.e.\ $X\perp\!\!\!\perp Z\mid Y$), then
\begin{equation}
  I_K(X;Y)\ge I_K(X;Z).
\end{equation}
In particular, this holds for one-dimensional Laplace pullback kernels by Corollary~\ref{cor:laplace-pullback-inference}.
\end{corollary}

\begin{proof}
Since $X\perp\!\!\!\perp Z\mid Y$, the posterior $\mu_{X\mid Z=z}$ is the $\mu_{Y\mid Z=z}$-mixture $\int \mu_{X\mid Y=y}\,d\mu_{Y\mid Z=z}(y)$.
By Lemma~\ref{lem:posterior-mixture-jensen} applied within this conditional law,
\begin{equation}
  H_K(X\mid Z=z)\ge \int H_K(X\mid Y=y)\,d\mu_{Y\mid Z=z}(y).
\end{equation}
Averaging over $z\sim\mu_Z$ gives $H_K(X\mid Z)\ge H_K(X\mid Y)$, hence $I_K(X;Y)\ge I_K(X;Z)$.
\end{proof}

\subsection{Concavity boundaries}
\label{subsec:concavity-boundaries}

Corollary~\ref{cor:laplace-pullback-inference} gives a useful class where the inference monotonicity holds.
For general (``fuzzy'') kernels $K^{\mathsf{X}}$, however, the Shannon-style inequality
\begin{equation}
  H_{K^{\mathsf{X}}}(X\mid Y) \le H_{K^{\mathsf{X}}}(X)
\end{equation}
need not hold.
Such nonconcavity can already occur in dimension $3$; see \cite[Thm.~9]{bavaud_effective_entropy}.

In contrast, the binary case is known to be concave \cite[Thm.~8]{gallego_gait_supp}, hence $H_{K^{\mathsf{X}}}(X\mid Y)\le H_{K^{\mathsf{X}}}(X)$ for all joint laws.
Another reasonable candidate regime is where $K$ is symmetric positive definite and satisfies the multiplicative triangle inequality $K(x,z)\ge K(x,y)K(y,z)$.
GAIT conjectures concavity in this setting~\cite[Conj.~1]{gallego_gait}, but SPD+MTI is not sufficient in general: Appendix~\ref{sec:appendix-18x18-spd-mti-counterexample} gives an $18$-state example, disproving their Conjecture~1.

\section{Representation and Discrete Approximation}
\label{sec:representation}
	
We now return to the discrete/continuous relationship. We show that continuous similarity-sensitive entropy can be understood as a limit of discrete similarity-matrix approximations.

\subsection{Uniform representation}

\begin{theorem}[Uniform representation]
\label{thm:uniform-representation}
For a law $\mu$ on a standard Borel space $\Omega$, there is a measurable map
$\psi:[0,1]\to\Omega$ such that $\psi(U)\sim\mu$ when $U$ is uniform on $[0,1]$.
With $K^\psi(u,u')=K(\psi(u),\psi(u'))$, we have
\[
  H_{K^\psi}(U)=H_K(\mu).
\]
If $\mu$ is atomless, $\psi$ may be chosen to be an isomorphism.
\end{theorem}

\begin{proof}
The map $\psi$ exists by the uniform randomization lemma
\cite[Lem.~2.22, p.~34]{kallenberg_foundations}. The entropy identity follows
from the pullback identity. In the atomless case the probability spaces
are isomorphic modulo null sets~\cite[Thm.~9.4.7, p.~283]{bogachev_measure_theory_2}.
\end{proof}

\subsection{Block averages and discrete approximations}
\label{sec:discrete-approximations}

This last theorem represents the input law by a uniform draw and puts its
similarities on $[0,1]$, where we can approximate them by block averages.

Divide $[0,1]$ into $n$ intervals $I_i$ of length $1/n$. Each interval
represents one state of a uniform law $p^{(n)}$. Average the kernel over
each pair of intervals to obtain
\[
  \widetilde K^{(n)}_{ij}
  = n^2\int_{I_i\times I_j} K(u,v)\,du\,dv.
\]
Let $K^{(n)}$ be the matrix obtained by setting the diagonal of
$\widetilde K^{(n)}$ to one.

\begin{theorem}[Discrete approximations to $H_K$]
\label{thm:discrete-to-continuous}
For a similarity kernel $K$ on $[0,1]$ with the uniform law $\lambda$,
\[
  H_{K^{(n)}}(p^{(n)})\longrightarrow H_K(\lambda).
\]
Evaluating the same entropy formula on the raw block averages
$\widetilde K^{(n)}$ gives the same limit. If the typicality of $K$ is
at least $\varepsilon>0$ almost everywhere, then
\[
  0\le H_{\widetilde K^{(n)}}(p^{(n)})-H_{K^{(n)}}(p^{(n)})
  \le \frac{1}{\varepsilon n}.
\]
\end{theorem}

\begin{proof}
Write $\tau$ for the typicality of $K$ under $\lambda$. The typicality
$t_n$ of the block-average matrix, read as a function on $[0,1]$, is
\[
  t_n(u)=(\widetilde K^{(n)}p^{(n)})_i
       =n\int_{I_i}\tau(v)\,dv,\qquad u\in I_i.
\]
The corresponding typicality $s_n$ after setting the diagonal to one
satisfies $0\le s_n-t_n\le 1/n$. By Lebesgue differentiation, both
$t_n$ and $s_n$ converge to $\tau$ almost everywhere.

Jensen's inequality on each interval gives
\[
  H_{K^{(n)}}(p^{(n)})\le H_{\widetilde K^{(n)}}(p^{(n)})
  =\int_0^1-\log t_n(u)\,du\le H_K(\lambda).
\]
Fatou's lemma applied to $-\log s_n$ gives
$H_K(\lambda)\le\liminf_n H_{K^{(n)}}(p^{(n)})$, so both sequences
converge to $H_K(\lambda)$.
Under the stated lower bound, $t_n\ge\varepsilon$ and
$\log(s_n/t_n)\le (s_n-t_n)/t_n\le 1/(\varepsilon n)$. Integrating gives
the stated bound.
\end{proof}

Together with uniform representation, this shows that entropies on
standard Borel spaces can be approximated by entropies of finite uniform
laws with similarity matrices.

\section{Discussion and Future Directions}

\subsection{Interpretation and Related Work}
\label{sec:discussion}

\paragraph*{Choices of similarity transport.}
For a fixed joint law, the similarity bound in \eqref{eq:transport-domination}
ensures that entropy cannot increase. Another way to compare two outputs is
to average the similarity of independent inputs drawn from their conditional laws. The resulting
typicality is the conditional mean of input typicality, so Jensen's inequality
also bounds the averaged entropy by the input entropy. The average compares
input distributions and can give a self-similarity below $1$. For example,
merging two equally likely, distinguishable inputs into one output gives
an average self-similarity of $1/2$. The entropy formula then gives $\log 2$
for a constant output, reflecting the variation among inputs behind that
output. A similarity kernel on the output states assigns similarity $1$
to each state paired with itself.

\paragraph*{Nuisance variables and conditional independence.}
If an observation is $Y=(Y_0,N)$ and the nuisance coordinate $N$ carries no extra information about $X$ beyond $Y_0$, then under the condition in Corollary~\ref{cor:cond-indep-noise} the induced output kernel factors through $Y_0$, and
\begin{equation}
  H_{K^{\mathsf{Y},\mu}}(Y)=H_{K^{\mathsf{Y}_0,\mu}}(Y_0).
\end{equation}
This is worth contrasting with the Shannon regime. On the inference side, Shannon mutual information already ignores such nuisance coordinates:
\begin{equation}
  I(X;Y_0,N)=I(X;Y_0).
\end{equation}
But plain Shannon output entropy does not:
\begin{equation}
  H(Y_0,N)=H(Y_0)+H(N\mid Y_0),
\end{equation}
so appending irrelevant noise can raise the entropy of the representation even when it adds no extra ability to distinguish the task states.

\paragraph*{Exact and noisy group labels.}
Let $Z=\pi(X)$ be the group label for a partition kernel $K$. As a
representation, $Z$ preserves every distinction encoded by $K$, so
$H_{K^{\mathsf{Z},\mu}}(Z)=H(Z)=H_K(X)$. As an observation, $Z$ determines
the group exactly, so Proposition~\ref{prop:conditional-partition-kernel}
gives $I_K(X;Z)=H(Z)$.

Now suppose $Z$ is a fair coin flip and $Y$ is obtained by flipping it
independently with probability $\varepsilon\in(0,1)$, where
$\varepsilon\ne 1/2$. The channel can merge the two groups into either
output label. Under the conservative semantics of the max rule, we count
a distinction as preserved only when the output labels guarantee it.
Here the distinction between the groups is lost, so
$H_{K^{\mathsf{Y},\mu}}(Y)=0$.

If our goal is to predict $X$, we use $Y$ as evidence about it. Observing
a label still makes one group more likely, so $I_K(X;Y)=I(Z;Y)>0$.

\paragraph*{Differential entropy and changing coordinates.}
Differential entropy can also be put in terms of the kernel. Let $K_\epsilon$ be the width-$\epsilon$ partition kernel on $\mathbb{R}$,
\begin{equation}
  K_\epsilon(x,x'):=\mathbf{1}\{\lfloor x/\epsilon\rfloor=\lfloor x'/\epsilon\rfloor\}.
\end{equation}
Then $H_{K_\epsilon}(X)=h(X)+\log(1/\epsilon)+o(1)$; see, e.g., \cite{cover_thomas}. If $Y=g(X)$ for a smooth bijection $g$, transporting this kernel gives
\begin{equation}
  H_{K_\epsilon^{g^{-1}}}(Y)=H_{K_\epsilon}(X).
\end{equation}
The transported partition on the $Y$-space is not uniform: the image of an $X$-bin of width $\epsilon$ has width about $|g'(x)|\epsilon$ near $y=g(x)$. But if one forgets the transported kernel and instead re-bins $Y$ into width-$\epsilon$ intervals, equivalently using the identity kernel on the new bin labels, then the local refinement term changes from $\log(1/(|g'(x)|\epsilon))$ to $\log(1/\epsilon)$, adding $\log|g'(x)|$. Averaging over $X$ recovers the usual formula
\begin{equation}
  h(Y)=h(X)+\mathbb{E}\log |g'(X)|.
\end{equation}
So the Jacobian term comes from resetting the output notion of proximity rather than transporting it from the original task. In the present framework, relabelings instead carry the kernel with them, and Proposition~\ref{prop:invariance-isomorphism} gives the corresponding coordinate-free equality
\begin{equation}
  H_{K'}(g_\#\mu)=H_K(\mu)
\end{equation}
whenever $g$ is an isomorphism from $(\Omega,\mu,K)$ to $(\Omega',g_\#\mu,K')$.

\paragraph*{Approximation, estimation, and finite similarity matrices.}
The uniform representation theorem and the step-kernel approximation theorem show how the continuous case is inside the same framework as the finite similarity-matrix case: continuous $K$-entropy is obtained as a limit of finite uniform similarity-matrix approximations of the same kernelled task. One can approximate $H_K(\mu)$ by finite matrices without discretizing the state space $X$, and one can simplify the task by coarsening $K$ without changing the underlying probability law. These are different choices, and keeping them separate avoids making arbitrary coarse-graining decisions by choice of the state space itself. Unlike coordinate binning for Shannon entropy, or density-based estimation of differential entropy, the approximation here targets the same similarity-sensitive quantity throughout.

This also suggests a direct empirical estimator. Given i.i.d.\ samples $\omega_1,\dots,\omega_n\sim\mu$, define
\begin{equation}
  \widehat\tau_i:=\frac1n\sum_{j=1}^n K(\omega_i,\omega_j),
  \qquad
  \widehat H_{K,n}:= -\frac1n\sum_{i=1}^n \log \widehat\tau_i.
\end{equation}
This is the discrete $K$-entropy of the empirical law. It avoids density estimation and is natural when pairwise similarities are more informative than coordinates.

\paragraph*{Relation to other viewpoints.}
Relative to classical information theory, the main difference is that the similarity kernel is treated as part of the probabilistic model itself. This viewpoint is close in spirit to Blackwell's ordering of experiments~\cite{blackwell_1953}: post-processing should not create useful information about the state, but one has to say what counts as ``useful.'' Here that happens through $K$. It is also close to other similarity-based entropy constructions, but the role played by the similarity matrix is not always the same.

Bavaud~\cite{bavaud_effective_entropy} studies the same discrete "reduced entropy" functional $p\mapsto H_K(p)$, so the nonconcavity examples there are directly relevant here. But the \emph{effective entropy} introduced in that work uses the $K$ as a confusion matrix rather than as task semantics, and it is not computed directly from the one-shot typicality profile $Kp$ (or $K\mu$). We use Bavaud's examples of concavity failure for the Leinster--Cobbold functional.

GAIT~\cite{gallego_gait,gallego_gait_supp} develops symmetric conditional and mutual quantities that equip both variables with kernels. Our emphasis is different: for inference about a fixed task variable we keep the kernel on $X$ fixed and think of $Y$ as an observation about $X$, whereas for representation change we transport semantics to the output and study induced-kernel entropy there.

\subsection{Further Directions}

\paragraph*{Concavity and conditional inequalities.}
The concavity landscape is summarized as follows:
\begin{itemize}
\item \emph{Identity / partition kernels}: $H_K$ reduces to Shannon entropy of the coarse variable and is concave (Proposition~\ref{prop:conditional-partition-kernel}).
\item \emph{$|\mathsf{X}|=2$}: $H_K$ is concave for every similarity matrix (see \cite[Thm.~8]{gallego_gait_supp}).
\item \emph{One-dimensional Laplace pullback kernels ($K(x,x')=\exp(-|g(x)-g(x')|)$)}: $H_K$ is concave on its domain (Theorem~\ref{thm:laplace-concavity}, Corollary~\ref{cor:laplace-pullback-inference}).
\item \emph{General fuzzy kernels, $|\mathsf{X}|\ge 3$}: nonconcavity can occur~\cite[Thm.~9]{bavaud_effective_entropy}, and SPD+MTI does not prevent it (Appendix~\ref{sec:appendix-18x18-spd-mti-counterexample}).
\end{itemize}
We do not know the minimal $|\mathsf{X}|$ for which SPD+MTI can fail concavity; the appendix gives an $18$-state example, but a smaller counterexample likely exists.
An open direction is to identify tractable sufficient conditions on
$K$ that guarantee concavity, or weaker hypotheses ensuring $I_K(X;Y)\ge 0$ for
restricted observation models and input laws.

\paragraph*{Scoring rules.}
When $H_K$ is concave and differentiable, its tangent gives a proper
scoring loss~\cite{gneiting_raftery}. Reporting the true law minimizes
expected loss, and that minimum is $H_K$. Updating the report from the
prior to the posterior therefore reduces expected prediction loss by
$I_K(X;Y)$ on average.
A companion essay derives the score~\cite{miller_scoring_rules}.

\paragraph*{Concavity under transport (inference after representation change).}
A natural question is how concavity interacts with transport under deterministic maps or channels.
Even when the source kernel is a one-dimensional Laplace pullback, $K(\omega,\omega')=\exp(-|g(\omega)-g(\omega')|)$, the induced output kernel may take a form like
$K^{\mathsf{Y},\mu}(y,y')=\exp(-d_g(y,y'))$ with $d_g$ an essential-infimum ``fiber distance'' between posterior laws (or just fibers). In general $d_g$ need not be representable as $|\widetilde g(y)-\widetilde g(y')|$ for a real-valued output feature $\widetilde g$, and the induced kernel need not remain in a concave class.
Characterizing maps/channels (or conditions on fibers/posteriors) under which induced kernels preserve concavity, or establishing weaker conditions guaranteeing $I_K\ge 0$ on the output space, is open.

\paragraph*{Design utility}
In Bayesian optimal experiment design, each design choice $a$ specifies an observation channel $P_a(\cdot\mid t)$ from a latent task variable $T\sim\mu$ to data $Y_a$, and hence a posterior law $\mu(\cdot\mid Y_a)$.
Fixing a similarity kernel $K^T$ on $T$, one can score $a$ by the expected information gain
$U(a):=H_{K^T}(\mu)-\mathbb{E}\!\left[H_{K^T}(\mu(\cdot\mid Y_a))\right]=I_{K^T}(T;Y_a)$.
One open problem is to understand when coarse surrogates---e.g.\ approximating $P_a$ or replacing $T$ by a coarse representation---preserve or approximate the ranking of designs under $U(a)$, and to bound the resulting error via induced-kernel coarse-graining inequalities.

\appendix

\section*{Proofs and Technical Lemmas}

\section{Concavity and mixtures of posteriors}
\label{sec:appendix-posterior-mixtures}

The entropy domain is closed under finite mixtures. For
$\nu=\sum_i w_i\nu_i$ with positive weights, its typicality satisfies
$\tau_\nu\ge w_i\tau_{\nu_i}$ for each $i$. Hence
$H_K(\nu)\le\sum_i w_i\bigl(H_K(\nu_i)-\log w_i\bigr)$.

\begin{lemma}[Concavity under posterior mixtures]
\label{lem:posterior-mixture-jensen}
If $H_K$ is concave on its domain, then
$H_K(X)\ge\mathbb E[H_K(\mu_{X\mid Y})]$.
\end{lemma}

\begin{proof}
Write $\mu_y=\mu_{X\mid Y=y}$ and $\mu=\mu_X$. Draw
$Y_1,Y_2,\dots$ independently with the law of $Y$. Mix the average
of the first $n$ posteriors with the prior:
\[
  \nu_n=\tfrac12\mu+\frac{1}{2n}\sum_{i=1}^n\mu_{Y_i}.
\]
By concavity, $\mathbb E H_K(\nu_n)\ge
\tfrac12 H_K(\mu)+\tfrac12\mathbb E H_K(\mu_{X\mid Y})$.

By exchangeability, $\mathbb E H_K(\nu_n)$ is the average of two
expectations of $-\log\tau_{\nu_n}(x)$. In one, draw $x$ from $\mu$
independently of the observations. In the other, draw $x$ from
$\mu_{Y_1}$ given $Y_1$, independently of the other observations.
In both cases $x$ has marginal law $\mu$. Given $x$, the typicalities
$\tau_{\mu_{Y_i}}(x)$ for $i\ge2$ are independent with mean
$\tau_\mu(x)$, so the law of large numbers gives
$\tau_{\nu_n}(x)\to\tau_\mu(x)$ almost surely. The prior half ensures
$\tau_{\nu_n}(x)\ge\tau_\mu(x)/2$, so $-\log\tau_{\nu_n}(x)$ is
bounded by the integrable $\log 2-\log\tau_\mu(x)$. Dominated convergence
gives $\mathbb E H_K(\nu_n)\to H_K(\mu)$. The concavity bound then gives
the claim.
\end{proof}

\section{Laplace concavity}
\label{sec:appendix-laplace-concavity-proof}

\subsection{Ordered grids}

Fix ordered grid points $x_1<\cdots<x_n$ and let
$K_{ij}=\exp(-|x_i-x_j|)$. For $n=1$ the result is immediate. For
$n\ge2$, let $\rho_i=\exp(-(x_{i+1}-x_i))$ and define the triangular
matrix $D$ by
\[
  (Du)_1=u_1,\qquad
  (Du)_{i+1}=\frac{u_{i+1}-\rho_i u_i}{\sqrt{1-\rho_i^2}}.
\]
Since $K_{i+1,j}=\rho_i K_{ij}$ for $j\le i$, the entries of
$DKD^\top$ off the diagonal are zero. Its diagonal entries are one,
so $L=K^{-1}=D^\top D$. The off-diagonal entries of $L$ are
nonpositive. The first and last row sums are $1/(1+\rho_1)$ and
$1/(1+\rho_{n-1})$, respectively. Each interior row sum is
$1/(1+\rho_{i-1})-\rho_i/(1+\rho_i)>0$.

For a probability vector $p$, write $q=Kp$ for its typicalities, which
are all positive. Let $r_i=\sum_j L_{ij}$ be the row sums and let
$w_{ij}=-L_{ij}$ for $i\ne j$. Since $p=Lq$,
\[
  p_i=r_iq_i+\sum_{j\ne i}w_{ij}(q_i-q_j).
\]
Substitute this into $-H_K(p)=\sum_i p_i\log q_i$ and group the terms
for each pair using $w_{ij}=w_{ji}$ to get
\[
  -H_K(p)=\sum_i r_i q_i\log q_i
          +\sum_{i<j}w_{ij}(q_i-q_j)\log\frac{q_i}{q_j}.
\]
The first sum is strictly convex in $q$ because $t\log t$ is strictly
convex and every $r_i$ is positive. The log-sum inequality gives joint
convexity of $a\log(a/b)$ for $a,b>0$. Thus
$(a-b)\log(a/b)=a\log(a/b)+b\log(b/a)$ is also jointly convex, so the
second sum is convex because $w_{ij}\ge0$. Since $q=Kp$ is linear and
injective, $H_K$ is strictly concave on the simplex.

\subsection{Extending to probability laws over the line}

Let $T_n$ clip the real line to $[-n,n]$ and round the result down to
the grid of spacing $1/n$. Then $T_n(x)\to x$. Clipping can only
shorten distances, and rounding moves each point by at most $1/n$, so
\[
  |T_n(x)-T_n(x')|\le |x-x'|+2/n.
\]
For any law $\mu$, put $\mu_n=(T_n)_\#\mu$. The typicality
$\tau_{\mu_n}(T_n(x))$ averages $\exp(-|T_n(x)-T_n(x')|)$ over
$x'\sim\mu$. By bounded convergence, it tends to $\tau_\mu(x)$.
The distance bound also gives
$\tau_{\mu_n}(T_n(x))\ge e^{-2/n}\tau_\mu(x)$. Thus
$-\log\tau_{\mu_n}(T_n(x))$ is bounded by $-\log\tau_\mu(x)+2$.
Dominated convergence under $\mu$ gives $H_{K_{\mathrm{Lap}}}(\mu_n)
\to H_{K_{\mathrm{Lap}}}(\mu)$.

The same $T_n$ works for two laws and their mixture. Pushforward
preserves the mixture weights, so applying grid concavity and taking
limits proves Theorem~\ref{thm:laplace-concavity}.

\section{Measurability of the induced output kernel}
\label{sec:appendix-measurability-induced-kernel}

This appendix records the measurability and disintegration facts used in
Proposition~\ref{prop:posterior-induced-fixed-law-minimality}.
Fix $q\in\mathbb{Q}\cap[0,1]$ and
$A_q:=\{(\omega,\omega')\in\Omega\times\Omega:\ K(\omega,\omega')>q\}$.
Since $y\mapsto\mu_{X\mid Y=y}$ is a probability kernel, a monotone-class argument
from rectangles shows that
$(y,y')\mapsto(\mu_{X\mid Y=y}\otimes\mu_{X\mid Y=y'})(A_q)$ is
$\mathcal{F}_Y\otimes\mathcal{F}_Y$-measurable.
For $y\neq y'$,
\begin{equation}
K^{\mathsf{Y},\mu}(y,y')
=
\sup_{q\in\mathbb{Q}\cap[0,1]}
q\,\mathbf{1}\{(\mu_{X\mid Y=y}\otimes\mu_{X\mid Y=y'})(A_q)>0\},
\end{equation}
so $K^{\mathsf{Y},\mu}$ is measurable as a supremum of countably many measurable
functions (with diagonal set to $1$ by convention).
Any two versions of the conditional law agree on a set of full
$\mu_Y$-measure. They give the same induced kernel on pairs from that set.

\section{Invariance under nuisance noise}
\label{sec:nuisance-invariance}

\begin{corollary}[A sufficient condition for nuisance-noise invariance]
\label{cor:cond-indep-noise}
Let $(X,Y)$ be jointly distributed, where $X$ takes values in the task space $\Omega$ equipped with similarity kernel $K$ and
$Y=(Y_0,N)$ takes values in $\mathsf{Y}=\mathsf{Y}_0\times\mathsf{N}$.
Assume that $X$ and $N$ are conditionally independent given $Y_0$, so that
for $\mu_Y$-almost every $(y_0,n)$,
\begin{equation}
  \mu_{X\mid (Y_0,N)=(y_0,n)} \;=\; \mu_{X\mid Y_0=y_0}.
  \label{eq:cond-indep-posterior}
\end{equation}
Assume also that
\begin{equation}
  \operatorname*{ess\,sup}_{(\omega,\omega')\sim \mu_{X\mid Y_0=y_0}^{\otimes 2}} K(\omega,\omega') = 1
  \quad\text{for }\mu_{Y_0}\text{-a.e.\ }y_0.
\end{equation}
Then under the output law $\mu_Y$,
\begin{equation}
K^{\mathsf{Y},\mu}\big((y_0,n),(y_0',n')\big)
=
K^{\mathsf{Y}_0,\mu}\big(y_0,y_0'\big),
\end{equation}
where $K^{\mathsf{Y},\mu}$ and $K^{\mathsf{Y}_0,\mu}$ are the output kernels
induced by $(X,Y)$ and $(X,Y_0)$, respectively, from the same input kernel $K$.
Consequently,
\begin{equation}
H_{K^{\mathsf{Y},\mu}}(\mu_Y)=H_{K^{\mathsf{Y}_0,\mu}}(\mu_{Y_0}).
\end{equation}
\end{corollary}

\begin{proof}
Write $\pi:\mathsf{Y}\to\mathsf{Y}_0$ for the projection $\pi(y_0,n):=y_0$.
Work on the set of full $\mu_Y$-measure where
\eqref{eq:cond-indep-posterior} holds and the essential supremum within
the conditional law at $\pi(y)$ is $1$.
For $\pi(y)\neq \pi(y')$, Definition~\ref{def:posterior-induced-output-kernel} gives
\[
  K^{\mathsf{Y},\mu}(y,y')
  =
  K^{\mathsf{Y}_0,\mu}(\pi(y),\pi(y')).
\]

On the diagonal of $\mathsf{Y}_0$, if $y=y'$ then both sides are $1$ by convention, while if $y\neq y'$ but $\pi(y)=\pi(y')=y_0$, the additional hypothesis gives
\[
  K^{\mathsf{Y},\mu}(y,y')
  =
  \operatorname*{ess\,sup}_{(\omega,\omega')\sim \mu_{X\mid Y_0=y_0}^{\otimes 2}} K(\omega,\omega')
  =
  1
  =
  K^{\mathsf{Y}_0,\mu}(y_0,y_0).
\]
Hence $K^{\mathsf{Y},\mu}=(K^{\mathsf{Y}_0,\mu})^\pi$, and the entropy
identity follows from Lemma~\ref{lem:pullback-identity-any-output-kernel}.
\end{proof}

\begin{remark}[Why the extra hypothesis is needed]
\label{rem:cond-indep-noise-atoms}
The extra hypothesis is only needed for pairs of distinct outputs
\((y_0,n)\neq (y_0,n')\) with the same value of \(y_0\). These pairs are
off-diagonal in \(\mathsf{Y}_0\times\mathsf{N}\), but they project to the diagonal
of \(\mathsf{Y}_0\), where the induced kernel is set equal to \(1\). The condition
ensures that the induced kernel on \(\mathsf{Y}_0\times\mathsf{N}\) also assigns
similarity \(1\) to such pairs. It can fail, for example, for
\(K(x,x')=c+(1-c)\mathbf 1\{x=x'\}\), with \(0<c<1\), when the posterior
law of \(X\mid Y_0=y_0\) is non-atomic. The essential supremum within that
posterior is then \(c\), although every diagonal entry is \(1\).
\end{remark}

\section*{Additional Results}

\section{An \texorpdfstring{$18\times 18$}{18x18} SPD+MTI kernel with nonconcave \texorpdfstring{$H_K$}{H\_K}}
\label{sec:appendix-18x18-spd-mti-counterexample}

This appendix gives a similarity matrix $K$ that is symmetric
positive definite (SPD) and satisfies the multiplicative triangle inequality (MTI), but for which
$p\mapsto H_K(p)$ fails concavity on the simplex.
This provides a negative resolution of GAIT Conjecture~1~\cite[Conj.~1]{gallego_gait}.

\begin{example}[An $18\times 18$ star-shaped kernel]\label{ex:spd-mti-18x18-hub-spoke}
Let $m:=17$, $a:=0.2$, and $b:=0.045$.
Index the states by $\{0,1,\dots,m\}$, with $0$ the center and $\{1,\dots,m\}$ the leaves, and define
\begin{equation}
  K_{00}=1,\qquad
  K_{0i}=K_{i0}=a\ (i\ge 1),\qquad
  K_{ii}=1,\qquad
  K_{ij}=b\ (i\neq j,\ i,j\ge 1).
\end{equation}
Equivalently, writing $\mathbf{1}\in\mathbb{R}^m$ for the all-ones vector,
\begin{equation}
  K=
  \begin{pmatrix}
    1 & a\mathbf{1}^\top\\
    a\mathbf{1} & (1-b)I_m+b\mathbf{1}\mathbf{1}^\top
  \end{pmatrix}.
\end{equation}

This matrix satisfies MTI ($K_{ik}\ge K_{ij}K_{jk}$ for all $i,j,k$): the cases with any repeated index or with $j$ a leaf and $i,k$ both leaves reduce to $b\ge b^2$ (true since $b<1$); when the hub is an endpoint ($i=0$, $j,k$ leaves or vice versa) the requirement is $a\ge ab$ (true); the binding case is distinct leaves $i,k$ with intermediate $j=0$, requiring $b\ge a^2$; here $0.045\ge 0.04$.
It is also SPD: since $1-b>0$, the Schur complement is
$(1-b)I_m+(b-a^2)\mathbf{1}\mathbf{1}^\top$, whose eigenvalues are $1-b$ (multiplicity $m-1$) and
$1-b+m(b-a^2)=1.04>0$.

It is nevertheless not concave in the sense that $p\mapsto H_K(p)$ fails concavity on $\Delta_{18}^\circ$.
Consider the one-parameter family of interior laws
\begin{equation}
  p(t):=\left(t,\frac{1-t}{m},\dots,\frac{1-t}{m}\right),\qquad t\in(0,1).
\end{equation}
Writing $\tau_0(t)$ for the typicality of the hub and $\tau_L(t)$ for the (common) typicality of the leaves,
\begin{equation}
  \tau_0(t)=t+a(1-t),\qquad
  \tau_L(t)=at+c(1-t),\qquad
  c:=\frac{1+(m-1)b}{m}.
\end{equation}
Hence
\begin{equation}
  H_K(p(t)) = -t\log\tau_0(t) - (1-t)\log\tau_L(t),
\end{equation}
Writing $f(t):=H_K(p(t))$ and abbreviating $\tau_0\equiv\tau_0(t)$, $\tau_L\equiv\tau_L(t)$,
\begin{equation}
  f'(t)=-\log\tau_0+\log\tau_L
        -\frac{t(1-a)}{\tau_0}
        +\frac{(1-t)(a-c)}{\tau_L}.
\end{equation}
Note the identities $\tau_0-(1-a)t=a$ and $\tau_L+(a-c)(1-t)=a$
(both follow from the definitions).
A second differentiation gives
\begin{equation}
  f''(t)=-(1-a)\frac{\tau_0+a}{\tau_0^2}
        +(a-c)\frac{\tau_L+a}{\tau_L^2}.
\end{equation}
Evaluating at $t=1$ (where $\tau_0(1)=1$ and $\tau_L(1)=a=0.2$) gives
\begin{equation}
  f''(1)=-(1-a)\frac{1+a}{1}+\frac{(a-c)\cdot 2a}{a^2}
        =-(1-a^2)+\frac{2(a-c)}{a}.
\end{equation}
Substituting $c=\frac{1+16\cdot 0.045}{17}=\frac{1.72}{17}$ so that
$a-c=0.2-\frac{1.72}{17}=\frac{1.68}{17}$,
\begin{equation}
  f''(1)=-0.96+\frac{2\cdot 1.68}{17\cdot 0.2}
        =-\frac{96}{100}+\frac{336}{340}
        =\frac{12}{425}>0,
\end{equation}
so $H_K$ is not concave on $\Delta_{18}^\circ$.
\end{example}

\section{Typicality Distribution and Partition Kernels}
\label{sec:structural-properties}

\begin{remark}[Transport equivalence (atomless case)]
\label{rem:transport-equivalence}
When $\mu$ is atomless, the isomorphisms in Theorem~\ref{thm:uniform-representation}
give uniform representations that differ by a measure-preserving relabeling
of $[0,1]$.
For kernels $K_1,K_2$ on $([0,1],\lambda)$ we write $K_1\sim K_2$ if there exists a measure-preserving isomorphism $T$ of $([0,1],\lambda)$ such that
\begin{equation}
K_2(u,u') = K_1(T^{-1}(u),T^{-1}(u')).
\end{equation}
We write $[K]_\sim$ for the transport-equivalence class of a kernel $K$ on $([0,1],\lambda)$.
For uniform representations obtained by isomorphisms from a fixed space,
the resulting kernels are transport-equivalent. Also $K_1\sim K_2$ if
and only if $([0,1],\lambda,K_1)$ and $([0,1],\lambda,K_2)$ are isomorphic.
\end{remark}

\begin{corollary}[An obstruction to finite partition equivalence]
\label{cor:necessary-condition-block-partition}
If a space with a similarity kernel is isomorphic to a finite-class
partition kernel with group masses $\alpha_1,\dots,\alpha_m$, then
the law of its typicality is $\sum_{j=1}^m\alpha_j\delta_{\alpha_j}$.
In particular, a typicality distribution with infinite support rules out
such an isomorphism.
\end{corollary}

\begin{proof}
Under a partition kernel, the typicality of a state in group $j$ is
$\alpha_j$, and the probability of drawing a state from that group is
$\alpha_j$. Summing the contributions from the groups gives the
law above, with equal masses contributing at the same value. Proposition~\ref{prop:invariance-isomorphism}
transfers this law across an isomorphism.
\end{proof}

\section*{Funding}
No external funding was received for this work.

\section*{Acknowledgements}
I want to thank Ross Granowski for many helpful conversations about the material in this paper over the years, and for comments on an early draft.

\section*{Data availability}
No new data were generated or analyzed in this study. The theoretical results and counterexample are fully specified in the manuscript.

\end{document}